\documentclass[12pt]{amsproc}
\usepackage{enumerate}
\usepackage{graphicx}
\usepackage{amssymb}
\usepackage{epstopdf}
\usepackage{amsmath,amsthm,amscd}
\usepackage{latexsym}
\usepackage{color}
\usepackage[colorlinks,citecolor=red,pagebackref,hypertexnames=false]{hyperref}
\usepackage{geometry}                
\numberwithin{equation}{section}

\theoremstyle{plain}
\newtheorem{theorem}{Theorem}[section]
\newtheorem{lemma}[theorem]{Lemma}
\newtheorem{corollary}[theorem]{Corollary}

\newtheorem*{Froslemma}{Frostman Lemma}

\theoremstyle{definition}

\newtheorem{case[theorem]}{Case}

\theoremstyle{remark}
\newtheorem{remark}[theorem]{Remark}

\numberwithin{equation}{section}

\def\dH{\dim_{{\mathcal H}}}
\def\R{\Bbb R}

\begin{document}

\title{An $L^2$-identity and pinned distance problem} 


\author{Bochen Liu}

\date{today}

\keywords{}

\email{Bochen.Liu1989@gmail.com}
\address{Department of Mathematics, Bar-Ilan University, Ramat Gan, Israel}

\thanks{The work is supported by ERC Starting Grant No. 713927}

\begin{abstract}
Let $\mu$ be a Frostman measure on $E\subset\R^d$. The spherical average estimate
$$\int_{S^{d-1}}|\widehat{\mu}(r\omega)|^2\,d\omega\lesssim r^{-\beta}  $$
was originally used to attack Falconer distance conjecture, via Mattila's integral. In this paper we consider the pinned distance problem, a stronger version of Falconer distance problem, and show that spherical average estimates imply the same dimensional threshold on both of them. In particular, with the best known spherical average estimates, we improve Peres-Schlag's result on pinned distance problem significantly.

The idea in our approach is to reduce the pinned distance problem to an integral where spherical averages apply. The key new ingredient is the following identity. Using a group action argument, we show that for any Schwartz function $f$ on $\mathbb{R}^d$ and any $x\in\mathbb{R}^d$,
$$\int_0^\infty |\omega_t*f(x)|^2\,t^{d-1}dt\,=\int_0^\infty|\widehat{\omega_r}*f(x)|^2\,r^{d-1}dr,$$
where $\omega_r$ is the normalized surface measure on $r S^{d-1}$. An interesting remark is that the right hand side can be easily seen equal to
$$c_d\int\left|D_x^{-\frac{d-1}{2}}e^{-2\pi i t\sqrt{-\Delta}}f(x)\right|^2\,dt=c_d'\int\left|D_x^{-\frac{d-2}{2}}e^{2\pi i t\Delta}f(x)\right|^2\,dt.$$
An alternative derivation of Mattila's integral via group actions is also given in the Appendix.
\end{abstract}
\maketitle
\section{Introduction}
\subsection{Falconer distance conjecture}
Given $E\subset\R^d$, $d\geq 2$, one can define its distance set as
$$ \Delta(E)=\{|x-y|: x,y \in E \}.$$
The famous Falconer distance conjecture (\cite{Fal85}) states that $\Delta(E)$ has full Hausdorff dimension, or even positive Lebesgue measure, whenever the Hausdorff dimension of $E$, denoted by $\dH(E)$, is greater than $\frac{d}{2}$. It is known (\cite{Fal85}) that $d\geq 2$ is necessary and the dimensional threshold $\frac{d}{2}$ is, up to the end point, optimal. This conjecture can be seen as a continuous version of the Erd\H{o}s distance problem, which has already been solved by Guth and Katz in the plane (\cite{GK15}).
\begin{theorem}
	[Guth, Katz, 2015]
	Suppose $P\subset\R^2$ is a finite set of $N$ points. Then for any $\epsilon>0$ there exists a constant $C_\epsilon>0$ such that
	$$\#(\Delta(P))\geq C_\epsilon N^{1-\epsilon}.  $$
\end{theorem}

Unlike the Erd\H{o}s distance problem, the Falconer distance conjecture is, however, far from being solved. Accumulating effort of different great mathematicians (see e.g., \cite{Fal85},\cite{Mat87},\cite{Sjo93},\cite{Bou94},\cite{Wol99},\cite{Erd05}), the best currently known results are due to Wolff  (\cite{Wol99}) in the plane and Erdo\~{g}an (\cite{Erd05}) in higher dimensions. They proved that $\Delta(E)$ has positive Lebesgue measure whenever $\dH(E)>\frac{d}{2}+\frac{1}{3}$. On the other hand, assuming $\dH(E)\geq 1$ in the plane, Bourgain \cite{Bou03} showed that there exists an absolute $\delta>0$ such that $\dH(\Delta(E))>\frac{1}{2}+\delta$. Very recently Keleti and Shmerkin (\cite{KS18}) improve Bourgain's result by showing that $\dH(\Delta(E))>\frac{1}{2}+\frac{5}{27}$ whenever $\dH(E)>1$ in the plane.

To obtain Wolff-Erdo\~{g}an's dimensional exponent (i.e., $\frac{d}{2}+\frac{1}{3}$), the following tool invented by Mattila \cite{Mat87} plays an important role. That is, to show that $\Delta(E)$ has positive Lebesgue measure, it suffices to prove that there exists a measure $\mu$ on $E$ such that
\begin{equation}\label{Mat} \mathcal{M}(\mu):=\int {\left( \int_{S^{d-1}} {|\widehat{\mu}(r \omega)|}^2 d\omega \right)}^2 r^{d-1} dr<\infty. \end{equation}

The following lemma provides a family of measures on $E$.
\begin{Froslemma}
	[see, e.g. \cite{Mat15}, Theorem 2.7]
	Suppose $E\subset\R^d$ and denote $\mathcal{H}^s$ as the $s$-dimensional Hausdorff measure. Then $\mathcal{H}^s(E)>0$ if and only if there exists a probability measure $\mu$ on $E$ such that 
$$\mu(B(x,r))\lesssim r^s $$
	for any $x\in\R^d$, $r>0$.
\end{Froslemma}

Since by definition $\dH(E)=\sup\{s:\mathcal{H}^s(E)>0\}$, Frostman Lemma implies that for any $s_\mu<\dH(E)$ there exists a probability measure $\mu$ on $E$ such that
\begin{equation}\label{Frostmanmeasure}\mu(B(x,r))\lesssim r^{s_\mu},\ \forall\ x\in\R^d,\ r>0.  
\end{equation}

In fact, the main results in Wolff's and Erdo\~{g}an's papers are, for any $\mu$ satisfying \eqref{Frostmanmeasure},
\begin{equation}\label{Wolff-Erdogan}
\int_{S^{d-1}} {|\widehat{\mu}(r \omega)|}^2 d\omega\lesssim_\epsilon r^{-\beta(s_\mu)+\epsilon},
\end{equation}
where
$$\beta(s)=\frac{d+2s-2}{4}, \ \  s\in[\frac{d}{2}, \frac{d}{2}+1].  $$

When $s$ is small, it is proved by Mattila (\cite{Mat87}) that \eqref{Wolff-Erdogan} holds with
\begin{equation*}\label{exp}
\beta(s)=
	\begin{cases}
	s, & s\in(0, \frac{d-1}{2}]\\
	\frac{d-1}{2}, & s\in [\frac{d-1}{2}, \frac{d}{2}]
	\end{cases}.
\end{equation*}
When $s$ is large, the best known result on \eqref{Wolff-Erdogan} is due to Luc\`{a} and Rogers (\cite{LR15}).
\vskip.125in
Plugging \eqref{Wolff-Erdogan} into \eqref{Mat}, it follows that
$$\mathcal{M}(\mu)\lesssim_\epsilon \int |\widehat{\mu}(\xi)|^2|\xi|^{-\beta(s_\mu)+\epsilon}\,d\xi,$$
which is known to be finite whenever $s_\mu+\beta(s_\mu)>d$ and $\epsilon$ small enough (see, e.g. \cite{Mat15}, Section 2.5). Solve it for $s_\mu$ to obtain $s_\mu>\frac{d}{2}+\frac{1}{3}$. 

\begin{remark}\label{Du's result}
	Shortly after this paper came out, \eqref{Wolff-Erdogan} was improved by Du-Guth-Ou-Wang-Wilson-Zhang (\cite{DGOWWZ18}) when $d\geq 3$ and later furthur improved by Du and Zhang (\cite{DZ18}) when $d\geq 4$. Now the best known dimensional exponent for Falconer distance conjecture is
	\begin{equation}
		\label{best-known}
		\dH(E)>\begin{cases}\frac{4}{3},&d=2\ (\text{Wolff})\\1.8,&d=3\ (\text{Du \emph{et al.}})\\\frac{d}{2}+\frac{1}{4}+\frac{1}{8d-4},&d\geq 4\ (\text{Du and Zhang})
			
		\end{cases}
	\end{equation}
\end{remark}

\subsection{Pinned distance problem}
A stronger version of the Falconer distance problem is the pinned distance problem, which states the following.
\vskip.25in
\emph{Pinned distance problem: How large the Hausdorff dimension of $E\subset\R^d$, $d\geq 2$ needs to be to ensure that there exists $x\in E$ such that the pinned distance set 
$$\Delta_x(E) = \{|x-y|:y\in E\}$$ has full Hausdorff dimension, or even positive Lebesgue measure? }
\vskip.25in
This problem was first studied by Peres and Schlag (\cite{PS00}). For $A\subset\R^d$, denote $|A|$ as its $d$-dimensional Lebesgue measure.

\begin{theorem}[Peres, Schlag, 2000]
Given $E\subset\R^d$, $d\geq 2$, then
\begin{equation}\label{PS}\dH( \{x\in\R^d: |\Delta_x(E)| = 0 \})\leq d+1-\dH(E).\end{equation}
In particular, if $\dH(E)>\frac{d+1}{2}$, there exists $x\in E$ such that $|\Delta_x(E)|>0$ .
\end{theorem}
Later this problem was studied by different authors (\cite{OO15}, \cite{ITU16}, \cite{Shm16}, \cite{IL17}, \cite{shm17}, \cite{KS18}) and the estimate \eqref{PS} was recently improved by Iosevich and the author  (\cite{IL17}) when $\dH(E)>\frac{d+1}{2}$. However, for the pinned distance problem, the best known dimensional exponent is still Peres-Schlag's $\frac{d+1}{2}$. There are also some results on special classes of sets. For example, Keleti and Shmerkin (\cite{KS18}) proved that for planar sets $E\subset\R^2$, $\Delta_x(E)$ has full Hausdorff dimension for some $x\in E$ if $\dH(E)>1$ and $\dim_P(E)\leq 2\dH(E)-1$, where $\dim_P$ denotes the packing dimension. One can also see \cite{Orp17}, \cite{Shm16}, \cite{shm17}.
\vskip.125in
As we can see, there is a gap on the known dimensional threshold between Falconer distance problem and pinned distance problem. So it is very natural to ask if the exponent for Falconer distance problem is also sufficient for the pinned distance problem. This is the main result of this paper.

\begin{theorem}
	\label{pin-dist}
	Suppose $E\subset\R^d$, $d\geq 2$. Assuming \eqref{Wolff-Erdogan}, then 
$$\dH( \{x\in\R^d: |\Delta_x(E)| = 0 \})\leq \inf\{s:\dH(E)+\beta(s)>d\}.$$
In particular, if $\dH(E)+\beta(\dH(E))>d$, there exists $x\in E$ such that $\Delta_x(E)$ has positive Lebesgue measure.
\end{theorem}
Wolff-Erdo\~{g}an's estimate on \eqref{Wolff-Erdogan} implies the following.
\begin{corollary}
	Suppose $E\subset\R^d$, $d\geq 2$. Then
$$\dH( \{x\in\R^d: |\Delta_x(E)| = 0 \})\leq \begin{cases}\frac{3}{2}d+1-2\dH(E), &\dH(E)\in[\frac{d}{2}, \frac{d+1}{2}]\\
	d-\dH(E), &\dH(E)>\frac{d+1}{2}
\end{cases}.$$
In particular, if $\dH(E)>\frac{d}{2}+\frac{1}{3}$, there exists $x\in E$ such that $\Delta_x(E)$ has positive Lebesgue measure.
\end{corollary}

The proof relies on an $L^2$-identity (see Section \ref{Sec-L^2-id}) and Wolff-Erdo\~{g}an's estimate (see Lemma \ref{lemma4}). Although Wolff-Erdo\~{g}an's estimate was originally used on the Falconer distance problem, this paper shows that it also helps on the pinned distance problem, where Mattila's integral \eqref{Mat} is replaced by a new integral (see \eqref{Mat-pin}).

\begin{remark}
	As we mentioned above in Remark \ref{Du's result}, the estimate \eqref{Wolff-Erdogan} has been improved by Du \emph{et al.} when $d\geq 3$ and later furthur improved by Du and Zhang when $d\geq 4$. Thus our dimensional exponent for the pinned distance problem has been improved to \eqref{best-known} as well.
\end{remark}
\subsection{Associated spherical means}\label{Sec-Sph-mean}
Let $\mu$ be a Frostman measure on $E\subset\R^d$ and $\omega_t$ be the normalized surface measure on $tS^{d-1}$. One can define a measure $\nu_x(t)$ on $\Delta_x(E)$ by
$$\int f(t)\, d\nu_x(t)=\int f(|x-y|)\,d\mu(y)=\lim_{\epsilon\rightarrow 0}\,c_d\int f(t)\left(\int_{S^{d-1}}\mu^\epsilon(x-t\omega)\,d\omega\right)\,t^{d-1}\,dt,$$
where $\mu^\epsilon\in C_0^\infty$, $\mu^\epsilon\rightarrow \mu$. Therefore as a distribution,
$$\nu_x(t)=\lim_{\epsilon\rightarrow 0}\,c_d\,t^{d-1}\,\omega_t*\mu^\epsilon(x).$$

To prove Theorem \ref{pin-dist}, a natural idea is to show that for any $F\subset\R^d$, $$\dH(F)>\begin{cases}\frac{3}{2}d+1-2\dH(E), &\dH(E)\in[\frac{d}{2}, \frac{d+1}{2}]\\
	d-\dH(E), &\dH(E)>\frac{d+1}{2}
\end{cases},$$
there must exist $x\in F$ such that the support of $\nu_x$ has positive Lebesgue measure. Thus it suffices to prove that there exists a measure $ \lambda$ on $F$ such that
$$\int \int_{t\approx 1} |\nu_x(t)|^2\,dt\,d \lambda(x) <\infty.$$
If it holds, the Radon-Nikodym derivatives $\frac{d\nu_x}{dt}\in L^2$ for $\lambda$-a.e. $x\in F$, which implies the support of $\nu_x$ has positive Lebesgue measure.
\vskip.125in
We shall prove a more general result. Define $$T_tf(x)=\omega_t*f(x),\ T^\mu_tf(x)=\omega_t*(f\,d\mu)(x),$$
$$||f||^2_{\dot{H}^s}=\int |\widehat{f}(\xi)|^2\,|\xi|^{2s}\,d\xi.  $$

\begin{theorem}\label{operator}Suppose $\lambda$ is a compactly supported measure satisfying \eqref{Frostmanmeasure}. Then
\begin{equation}\label{operator1}||T_t f||_{L^2(t^{d-1}dt\times d \lambda)}\lesssim_\epsilon ||f||_{\dot{H}^{-\frac{\beta(s_\lambda)}{2}+\epsilon}}.\end{equation}
In particular, if in addition $\mu$ satisfies \eqref{Frostmanmeasure} and $s_\mu+\beta(s_{ \lambda})>d$, then
\begin{equation}\label{operator2}||T^\mu_t f||_{L^2(t^{d-1}dt\times d \lambda)}\lesssim ||f||_{L^2{(\mu)}}.\end{equation}
\end{theorem}

As we explained above, Theorem \ref{pin-dist} follows from Theorem \ref{operator}. In fact the dimensional exponent in Theorem \ref{pin-dist} comes from solving $s_{ \lambda}$ from $s_\mu+\beta(s_{ \lambda})=d$.
\vskip.125in
A straightforward consequence of Theorem \ref{operator} is
\begin{equation}\label{int-chain}||T^\mu_{t_{k+1}}\circ\cdots\circ T^\mu_{t_1} f||_{L^2(t_1^{d-1}dt_1\times\cdots\times t_{k+1}^{d-1}dt_{k+1}\times d\mu)}\lesssim ||f||_{L^2(\mu)},\ \text{if}\ s_\mu+\beta(s_\mu)>d.\end{equation}
To see this, take $\lambda=\mu$, fix $t_1,\dots t_k$ and apply \eqref{operator2} to $T^\mu_{t_{k}}\circ\cdots\circ T^\mu_{t_1} f$. It follows that
$$||T^\mu_{t_{k+1}}\circ\cdots\circ T^\mu_{t_1} f||_{L^2(t_1^{d-1}dt_1\times\cdots\times t_{k+1}^{d-1}dt_{k+1}\times d\mu)}\lesssim ||T^\mu_{t_{k}}\circ\cdots\circ T^\mu_{t_1} f||_{L^2(t_1^{d-1}dt_1\times\cdots\times t_{k}^{d-1}dt_{k}\times d\mu)},$$
if $s_\mu+\beta(s_\mu)>d$, which implies \eqref{int-chain} by induction. The following geometric result then follows from \eqref{int-chain}.
\begin{corollary}
	Suppose $E\subset\R^d$, $d\geq 2$, $\dH(E)+\beta(\dH(E))>d$. Then for any $k\in\mathbb{Z}^+$, the $k$-chain set,
	$$\{(|x_1-x_2|,\dots,|x_{k}-x_{k+1}|): x_j\in E\}  $$
	has positive $k$-dimensional Lebesgue measure.
\end{corollary}
By the results of Wolff, Du \emph{et al.} and Du-Zhang, this corollary holds whenever \eqref{best-known} holds. This improves results in \cite{BIT16}, where $s_\mu>\frac{d+1}{2}$ is obtained, and results in \cite{Liu17chains}, where only $k=2$ is considered.

\subsection{An $L^2$-identity and weighted Strichartz estimates}\label{Sec-L^2-id}
The key new ingredient in this paper is the following $L^2$-identity. Denote $d\omega_r$ as the normalized surface measure on $rS^{d-1}$. Also denote $d\omega=d\omega_1$.
\begin{theorem}
	\label{identity}
	For any Schwartz function $f$ on $\R^d$, $d\geq 2$ and any $x\in\R^d$,
$$\int_0^\infty |\omega_t*f(x)|^2\,t^{d-1}dt\,=\int_0^\infty|\widehat{\omega_r}*f(x)|^2\,r^{d-1}dr.$$
\end{theorem}

This identity links the spherical mean value operator (on $f$) and the extension operator (on $\widehat{f}$), where restriction estimates apply. Moreover, the right hand side equals, by Plancherel,
\begin{equation*}\begin{aligned}\int\left|\int  e^{-2\pi i tr}\left(\widehat{f}\,d\omega_r\right)^\vee(x)\,r^{\frac{d-1}{2}}dr\right|^2\,dt= &\int\left|\iint  e^{-2\pi i tr}\,e^{2\pi i x\cdot r\omega}\widehat{f}(r\omega)\,d\omega\,r^{\frac{d-1}{2}}dr\right|^2\,dt\\=&c_d\int\left|\iint  e^{-2\pi i t|\xi|}\,e^{2\pi i x\cdot \xi}\widehat{f}(\xi)\,|\xi|^{-\frac{d-1}{2}}d\xi\right|^2\,dt\\=&c_d\int\left|D_x^{-\frac{d-1}{2}}e^{-2\pi i t\sqrt{-\Delta}}f(x)\right|^2\,dt,
\end{aligned}
\end{equation*}
where $(\,)^\vee$ denotes the inverse Fourier transform, $D_x^\alpha=(-\Delta)^{\frac{\alpha}{2}}$ and $\Delta$ is the standard Laplacian. Similarly, with $r'=r^2$, it follows that
$$\int|\widehat{\omega_r}*f(x)|^2\,r^{d-1}dr= \frac{1}{2}\int|\widehat{\omega_{\sqrt{r'}}}*f(x)|^2\,(r')^{\frac{d-2}{2}}dr' = c_d'\int\left|D_x^{-\frac{d-2}{2}}e^{2\pi i t\Delta}f(x)\right|^2\,dt. $$

Therefore, the norm $||T_t f||^2_{L^2(t^{d-1}dt\times d \lambda)}$ in Theorem \ref{operator} is, in fact,
\begin{equation}\label{Mat-pin}
	\begin{aligned}
		\iint |\omega_t*f(x)|^2\,t^{d-1}dt\,d\lambda(x)=&\iint |\widehat{\omega_r}*f(x)|^2\,r^{d-1}dr\,d\lambda(x) \\=&c_d\iint\left|D_x^{-\frac{d-1}{2}}e^{-2\pi i t\sqrt{-\Delta}}f(x)\right|^2\,dt\,d\lambda(x)\\=&c_d'\iint\left|D_x^{-\frac{d-2}{2}}e^{2\pi i t\Delta}f(x)\right|^2\,dt\,d\lambda(x).
	\end{aligned}
\end{equation}
In other words, we reduce the pinned distance problem to weighted $L^2$-estimates for the wave (or Schr\"{o}dinger) operator. This kind of estimates was first studied by Ruiz and Vega in \cite{RV94}, where they investigate perturbations of the free equation by time-dependent potentials. More precisely they consider 
$$\iint\left|e^{-2\pi i t\sqrt{-\Delta}}f(x)\right|^2\,V(t,x)\,dt\,dx,\ \ \iint\left|e^{2\pi i t\Delta}f(x)\right|^2\,V(t,x)\,dt\,dx,  $$
where $\sup_{t>0} V\in \mathcal{L}^{\alpha,p}$, the Morrey-Campanato classes, defined by
$$||w||_{\mathcal{L}^{\alpha,p}}=\sup_{r, x_0}r^{\alpha}\left(r^{-d}\int_{B(x_0,r)}|w(x)|^p\,dx\right)^\frac{1}{p}<\infty.  $$
One can also see \cite{BBCRV08}, \cite{BBCRV10}, \cite{Korean16} for related work. An explicit weight, $|x|^s\,dx\,dt$, is discussed in \cite{HMSSZ10} (see (2.6) there). Although, unfortunately, none of their results helps in the distance problem, it is interesting to see this connection between geometric measure theory and PDE.

Another remark is, the identity in Theorem \ref{identity} is related to Kaneko-Sunouchi's work in 1985 (\cite{KS85}), where they show pointwise equivalence between square functions generated by generalized spherical means and Bochner-Riesz means. The author would like to thank Anthony Carbery to point it out.

\vskip.125in
{\bf Notation.} $X\lesssim Y$ means $X\leq CY$ for some constant $C>0$. $X\lesssim_\epsilon Y$ means $X\leq C_\epsilon Y$ for some constant $C_\epsilon>0$, depending on $\epsilon$.

For any set $A\subset\R^d$, $|A|$ denotes its $d$-dimensional Lebesgue measure.

Denote $d\omega_r$ as the normalized surface measure on $rS^{d-1}$. Also denote $d\omega=d\omega_1$.

$\widehat{f}(\xi):=\int e^{-2\pi i x\cdot\xi} f(x)\,dx$ is the Fourier transform and $f^\vee(\xi):=\int e^{2\pi i x\cdot\xi} f(x)\,dx$ is the inverse Fourier transform.

$||f||^2_{\dot{H}^s}:=\int |\widehat{f}(\xi)|^2\,|\xi|^{2s}\,d\xi$.

Denote $\Delta$ as the standard Laplacian and $D_x^\alpha=(-\Delta)^{\frac{\alpha}{2}}$.

\section{Proof of Theorem \ref{identity}}
The proof relies on a group action argument. The idea of using group action argument to attack distance problem dates back to the solution to the Erd\H{o}s distance problem (\cite{ES11}, \cite{GK15}). On Falconer distance problem, authors in \cite{GILP15} observed that Mattila's integral \eqref{Mat} can be interpreted in terms of Haar measures on $O(d)$, that is,
$$\int {\left( \int_{S^{d-1}} {|\widehat{\mu}(r \omega)|}^2 d\omega \right)}^2 r^{d-1} dr=\int {|\widehat{\mu}(\xi)|}^2\left( \int_{O(d)}{|\widehat{\mu}(\theta \xi)|}^2 d\theta \right)\,d\xi.  $$
In this paper we follow the idea in \cite{Liu17}, where an alternative derivation of Mattila's integral \eqref{Mat} is given (see Appendix). Similar reduction can also be found in \cite{Liu17chains}.
\vskip.125in
We may assume $f$ is real. Denote $d\theta$ as the Haar measure on $O(d)$, the orthogonal group. By the invariance of the Haar measure, we can write
$$T_t f (x)= \omega_t*f(x) = \int_{S^{d-1}} f(x-t\omega)\,d\omega = \int_{O(d)} f(x-t\theta \omega_0)\,d\theta, $$
where $\omega_0\in S^{d-1}$ is arbitrary but fixed. Then 
\begin{equation*}
\begin{aligned}
\int |T_t f(x)|^2\,t^{d-1}dt=&\int\left(\int_{S^{d-1}} f(x-t\omega)\,d\omega\right)\left(\int_{O(d)} f(x-t\theta \omega_0)\,d\theta\right)t^{d-1}dt\\=&\int\int_{S^{d-1}} f(x-t\omega)\left(\int_{O(d)} f(x-t\theta \omega_0)\,d\theta\right)\,d\omega\,t^{d-1}dt.
\end{aligned}
\end{equation*}
By the invariance of the Haar measure, we may replace $\omega_0$ by $\omega$. By polar coordinates $y=t\omega$, we have $d\omega\,t^{d-1}dt=\frac{1}{|S^{d-1}|}\, dy$. It follows that

\begin{equation*}
\begin{aligned}
&\int |T_t f(x)|^2\,t^{d-1}dt\\=&\frac{1}{|S^{d-1}|}\int f(x-y)\left(\int_{O(d)} f(x-\theta y)\,d\theta\right)\,dy\\=&\frac{1}{|S^{d-1}|}\int_{O(d)} \left(\int f(x-y)\, f(x-\theta y)\,dy\right)\,d\theta\\=&\frac{1}{|S^{d-1}|}\int_{O(d)} \left(\int \widehat{f}(-\xi)e^{-2\pi i x\cdot\xi}\ \widehat{f}(\theta\xi)e^{2\pi i x\cdot\theta\xi}\,d\xi\right)\,d\theta\\=&\int \left(\int_{S^{d-1}}\int_{S^{d-1}} \widehat{f}(-r\omega)e^{-2\pi i x\cdot r\omega}\ \widehat{f}(r\omega')e^{2\pi i x\cdot r\omega'}\,d\omega\,d\omega'\right)\,r^{d-1}\,dr\\=&\int|\widehat{\overline{\widehat{f\,}}d\omega_r}(x)|^2\,r^{d-1}dr\\=&\int|f*\widehat{\omega_r}(x)|^2\,r^{d-1}dr,\end{aligned}\end{equation*}
as desired.

\section{Some lemmas on Frostman measures}
\begin{lemma}[\cite{BIT16}, Lemma 2.5]\label{lemma1}
Suppose $\mu$ satisfies \eqref{Frostmanmeasure}. Then
	$$\int_{|\xi|\leq R}|\widehat{f\,d\mu}(\xi)|^2\,d\xi\lesssim R^{d-s_\mu}||f||_{L^2(\mu)}^2. $$
\end{lemma}
We give the proof below for the sake of completeness.
\begin{proof}
Take $\psi\in C_0^\infty(\R^d)$ whose Fourier transform is positive on the unit ball. Then
\begin{equation*}
\begin{aligned}
	\int_{|\xi|\leq R}|\widehat{f\,d\mu}(\xi)|^2\,d\xi&\lesssim \int|\widehat{f\,d\mu}(\xi)|^2\,\widehat{\psi}(\frac{\xi}{R})\,d\xi\\& = R^d\iint \psi(R(x-y))\,f(x)f(y)d\mu(x)\,d\mu(y).
\end{aligned}	
\end{equation*}
Since $\psi$ has bounded support and $\mu$ satisfies \eqref{Frostmanmeasure},
$$\int|\psi(R(x-y))|\,d\mu(x)\lesssim R^{-s_\mu},\ \int|\psi(R(x-y))|\,d\mu(y)\lesssim R^{-s_\mu}.$$
Then the lemma follows by Shur's test.
\end{proof}
\begin{lemma}
	\label{lemma4}
	Suppose $\lambda$ is a compactly supported measure satisfying \eqref{Frostmanmeasure}. Then
$$\int |\widehat{g\,d\omega_R}|^2\,d\lambda\lesssim_\epsilon R^{-\beta(s_\lambda)+\epsilon} ||g||^2_{L^2(\omega_R)}. $$
\end{lemma}
\begin{proof}
Denote $A_R$ as the $1$-neighborhood of $RS^{d-1}$. In Wolff's and Erdo\~{g}an's proof of \eqref{Wolff-Erdogan}, what was proved is, for any $h$ supported on $A_R$,
\begin{equation}\label{Wolff-Erdogan1}\int |\widehat{h}|^2\,d\lambda\lesssim_\epsilon R^{d-1-\beta(s_\lambda)+\epsilon}||h||_{L^2(A_R)}^2.
\end{equation}
One can see, e.g. \cite{Mat15}, Chapter 16 for the reduction. 

In our case, since $\lambda$ has compact support, one can find $\phi\in C_0^\infty$ such that $|\widehat{\phi}|\geq 1$ on the support of $\lambda$. Then $(g\,d\omega_R)*\phi$ is smooth on $A_R$. Therefore by \eqref{Wolff-Erdogan1},
$$\int |\widehat{g\,d\omega_R}|^2d\lambda\leq \int |\widehat{g\,d\omega_R}|^2|\widehat{\phi}|^2d\lambda=\int |\widehat{(g\,d\omega_R)*\phi}|^2d\lambda\lesssim_\epsilon R^{d-1-\beta(s_\lambda)+\epsilon}\int_{A_R}|(g\,d\omega_R)*\phi|^2.$$
Since $\phi$ has compact support and $\omega_R$ is the normalized surface measure on $RS^{d-1}$,
\begin{equation*}\begin{aligned}&\int_{A_R} |(g\,d\omega_R)*\phi|^2 \\\lesssim &\int_{A_R}\left(\int_{RS^{d-1}}|\phi(x-y)||g(y)|^2\,d\omega_R(y) \right)\left(\int_{RS^{d-1}}|\phi(x-y)|\,d\omega_R(y)\right)dx\\\lesssim &R^{-d+1} \int_{A_R}\int_{RS^{d-1}}|\phi(x-y)||g(y)|^2\,d\omega_R(y)dx\\\lesssim& R^{-d+1}||g||^2_{L^2(\omega_R)},\end{aligned}\end{equation*}
as desired.
\end{proof}

\section{Proof of Theorem \ref{operator}}
By Theorem \ref{identity},
$$\iint |T_t f(x)|^2\,t^{d-1}dt\,d\lambda(x)=\iint|\widehat{\overline{\widehat{f\,}}d\omega_r}(x)|^2\,d\lambda(x)\,r^{d-1}dr.$$
Then by Lemma \ref{lemma4}, it is bounded above by
$$\int\int_{rS^{d-1}} |\widehat{f}|^2\,d\omega_r\,r^{-\beta(s_{ \lambda})+\epsilon}\,r^{d-1}\,dr= c_d\int|\widehat{f}(\xi)|^2\,|\xi|^{-\beta(s_{ \lambda})+\epsilon}d\xi,$$
which completes the proof of \eqref{operator1} in Theorem \ref{operator}. For \eqref{operator2}, it suffices to show, when $s_\mu+\beta(s_{ \lambda})>d$, 
$$\int|\widehat{f\,d\mu}(\xi)|^2\,|\xi|^{-\beta(s_{ \lambda})+\epsilon}d\xi\lesssim ||f||_{L^2(\mu)}^2.$$
To see this, by Lemma \ref{lemma1},
\begin{equation*}
\begin{aligned}\int|\widehat{f\,d\mu}(\xi)|^2\,|\xi|^{-\beta(s_{ \lambda})+\epsilon}d\xi\lesssim &\ 2^{j(-\beta(s_{ \lambda})+\epsilon)}\int_{|\xi|\approx 2^j}|\widehat{f\,d\mu}(\xi)|^2\,d\xi\\\lesssim&\sum_j 2^{j(-\beta(s_{ \lambda})+d-s_\mu+\epsilon)}||f||_{L^2(\mu)}^2,
\end{aligned}
\end{equation*}
which is $\lesssim ||f||_{L^2(\mu)}^2$ if $s_\mu+\beta(s_{ \lambda})>d$, as desired.
\vskip.25in
\appendix
\begin{center}
    {\bf APPENDIX: An alternative derivation of Mattila's integral \eqref{Mat} via group actions}
  \end{center}
We only sketch the proof. One can see \cite{Liu17} for details.
\vskip.125in
Denote $E(d)$ as the group of rigid motions, with Haar measure $dg$, and $\sigma_t$ as the normalized surface measure on $\{(x,y)\in\R^d\times\R^d:|x-y|=t\}$. Roughly speaking there are two ways to define a measure $\nu$ on the distance set $\Delta(E)$,
$$\nu(t)=\int_{|x-y|=t}\mu(x)\,\mu(y)\,d\sigma_t(x,y)=\int_{E(d)}\mu(gx_t)\,\mu(gy_t)\,dg,  $$
where $|x_t-y_t|=t$, arbitrary but fixed. Therefore for any $J(t)>0$,
$$\int |\nu(t)|^2\,J(t)dt= \int \left(\int_{|x-y|=t}\mu(x)\,\mu(y)\left(\int_{E(d)}\mu(gx_t)\,\mu(gy_t)\,dg\right)d\sigma_t(x,y) \right)J(t)dt $$

By the invariance of the Haar measure, we can take $x_t=x$, $y_t=y$. Choose $J(t)>0$ such that $d\sigma_t(x,y)\,J(t)dt=dx\,dy$. Now the integral equals
$$\int_{E(d)}\left|\int\mu(x)\mu(gx)\,dx\right|^2\,dg=\int_{O(d)}\int_{\R^d}\left|\int\mu(x)\mu(\theta x+z)\,dx\right|^2\,dz\,d\theta.$$
By Plancherel in $x$, it equals
$$\int_{O(d)}\int_{\R^d}\left|\int_{\R^d} \widehat{\mu}(\xi)\,\widehat{\mu}(\theta \xi)\,e^{2\pi i z\cdot\xi}\,d\xi\right|^{2}\,dz\,d\theta.$$ 
By Plancherel in $z$, it equals
$$\int_{O(d)} \left(\int_{\R^d}|\widehat{\mu}(\xi)|^2\,|\widehat{\mu}(\theta\xi)|^2\,d\xi\right)\,d\theta=c_d\int {\left( \int_{S^{d-1}} {|\widehat{\mu}(r \omega)|}^2 d\omega \right)}^2 r^{d-1} dr,$$
as desired.
\bibliographystyle{abbrv}
\bibliography{/Users/MacPro/Dropbox/Academic/paper/mybibtex.bib}

\end{document}